\documentclass{amsart}

\usepackage{amsmath}
\usepackage{amssymb}
\usepackage{amsthm}
\usepackage{amsfonts}
\usepackage{graphicx}
\usepackage{mathrsfs}
\usepackage{enumerate}

\newtheorem{theorem}{Theorem}[section]
\newtheorem{lemma}[theorem]{Lemma}

\theoremstyle{definition}
\newtheorem{definition}[theorem]{Definition}

\theoremstyle{remark}

\numberwithin{equation}{section}

\newtheorem{fac}{Fact}

\newtheorem{rem}{Remark}

\newcommand{\proc}{$(X,\mathscr P,\mu,T)$}

\newcommand{\inv}{invariant} 

\newcommand{\sq}{sequence}

\newcommand{\na}{\mathbb N}

\newcommand{\p}{\mathscr P}              
\newcommand{\Q}{\mathscr Q}              
\newcommand{\R}{\mathscr R}              
\newcommand{\xsmt}{$(X,\mathfrak B,\mu,T)$}
\newcommand{\xt}{$(X,T)$}
\newcommand{\xsm}{$(X,\mathfrak B,\mu)$}

\newcommand{\htop}{\mathbf h_{\mathsf{top}}}
\newcommand{\en}{entropy}

\newcommand{\tl}{topological}
\newcommand{\ds}{dynamical system}

\newcommand{\diam}{\mathsf{diam}}
\newcommand{\dist}{\mathsf{dist}}

\title{Positive topological entropy implies chaos DC2}
\date{July 21, 2011}

\author{T. Downarowicz}

\thanks{Research supported from resources 
for science in years 2009-2012 as research project (grant MENII N N201 394537, Poland)}



\begin{document}

\begin{abstract}
Using methods of entropy in ergodic theory, we prove as in the title.
\end{abstract}

\maketitle

\section{Introduction}
The notion of chaos was invented in 1975 by Li and Yorke in their seminal paper \cite{LY} in the context of continuous transformations
of the interval. Since then several refinements of chaos have been introduced and extensively studied, mainly in connection with
interval maps or variants of the Sharkovsky Theorem. 
There are in fact very few implications between chaos and other topological or measure-theoretic invariants that hold for general \tl\ \ds s (understood as continuous transformations of compact metric spaces); most of results are either counterexamples or require strong assumptions. Depending on the attitude, this fact may be accounted as a disadvantage of the notions of chaos (as somewhat artificial) or as their advantage (they capture new phenomena). We skip a wider historical review of the topic, as this is done in many papers and survey papers devoted to the subject, and restrict to the basic facts that concern the subject of this paper. We refer the reader to the survey \cite{BHS} for more information.

One of the few general recent affirmative results says that positive \tl\ entropy implies Li-Yorke chaos (\cite{BGKM}). That the converse need not hold was known since 1986 \cite{S1}.

A strengthening of the notion of chaos was introduced in 1994 by Schweizer and Sm\'ital (so called Schweizer--Sm\'ital chaos;
\cite{SS}). Since then it has evolved into three variants of so-called \emph{distributional chaos} DC1, DC2 and DC3 (ordered from strongest to weakest). Around 2000 Pikula showed that positive \tl\ entropy does not imply DC1 (published in 2007 \cite{P}). 
Sm\'ital conjectured that it does imply DC2 (see e.g. \cite{S2}). This problem has been around for several years and several 
papers propose partial solutions (see e.g. \cite{O1} for UPE systems or \cite{S2} for some other cases).

In this paper we solve the problem posed by Sm\'ital by proving his conjecture in the general case, in particular, we strengthen
the result of \cite{BGKM}:

\begin{theorem}\label{tw1}
Assume that a \tl\ \ds\ \xt\ has positive \tl\ \en\ $\htop(T)>0$. Then the system possesses 
an uncountable DC2-scrambled set. 
\end{theorem}
Additionally, the scrambled set can be obtained perfect, which resolves another conjecture of Sm\'ital (see final remarks of this paper).

\medskip
Recall that a set $E\subset X$ is \emph{DC2-scrambled} if so is any pair $x,y\in E$, $x\neq y$. It remains to recall the definition of a DC2-scrambled pair: 

\begin{definition}
A pair $x,y\in X$ is DC2-scrambled if the following two conditions hold:
\begin{align*}
\forall_{t>0}\ \ \limsup_{n\to\infty} &\tfrac1n \#\{i\in[0,n-1]: \dist(T^ix,T^iy)< t\} =1, \\
\exists_{t_0>0} \ \ \liminf_{n\to\infty} &\tfrac1n \#\{i\in[0,n-1]: \dist(T^ix,T^iy)< t_0\} <1, 
\end{align*}
in other words, the orbits of $x$ and $y$ are arbitrarily close with upper density one, but for some distance -- with lower density
smaller than one.
\end{definition}

The inspiration for the proof presented in this paper came from the interpretation of the definition. It is not hard to see that
the above two conditions are equivalent to the following ones:
$$
\liminf_{n\to\infty} \frac1n \sum_{i=1}^n \dist(T^ix,T^iy)=0 \ \ \text{\ \ and \ } \ \ \limsup_{n\to\infty} \frac1n \sum_{i=1}^n \dist(T^ix,T^iy)>0.
$$
(We leave the easy verification of the equivalence to the interested reader.)
In this manner DC2 is expressed in terms of ergodic averages (of a function on the product space), which makes it a very natural ergodic-theoretic object. (DC1 or DC3 do not translate so well, not to mention Li-Yorke chaos). Although we are not going to use this rephrased definition directly in the proof, it helped us realize that the key to the solution lies in ergodic theory and measure-theoretic entropy. Indeed, topological properties play a marginal role in the proof. 

Most of the time we will be concerned with an ergodic measure-theoretic dynamical system \xsmt, where \xsm\ is 
a standard probability space and $T:X\to X$ is a measure-preserving transformation (endomorphism). Any finite measurable partition 
$\p$ of $X$ generates a factor of the system obtained by projecting onto the subinvariant sub-sigma algebra $\p^{[0,\infty)}$ of 
$\mathfrak B$, where we use the following notational convention: for $D\subset\na\cup\{0\}$ we set
$$
\p^D = \bigvee_{n\in D}T^{-n}(\p).
$$ 
(if $D$ is finite this is a finite partition, otherwise it is a sigma-algebra). Abusing slightly the notation, we will denote such 
a factor by \proc\ and call it \emph{the process generated by $\p$}. This process is isomorphic to a symbolic system, in which $\p$ plays the role of the alphabet and in which $\mu$ becomes a shift-invariant measure.

This and the rest of our notation (which is relatively standard) matches roughly that in \cite{D}. Most of the facts concerning entropy of processes and other facts in ergodic theory and \tl\ dynamics can be found in that book; in most cases we provide statement reference.

\section{Some tools}

\subsection{Roughly equal partitions}
The Shannon--McMillan Theorem (Mc Millan, 1953) asserts, that given an ergodic process \proc\ with positive entropy $h=h_\mu(T,\p)$, and some $\varepsilon>0$, if $n$ is large enough then $1-\varepsilon$ of the space is covered by cylinders of length $n$ whose measures range within $2^{-n(h\pm\varepsilon)}$.
(In 1957 Breiman strengthened this result proving almost everywhere convergence of the corresponding information function; this version is known as the \emph{Shannon--McMillan--Breiman Theorem} and can be found in any textbook on ergodic theory.)

The above property of the partition $\p^{[0,n-1]}$ formed by the cylinders of length $n$ becomes the main tool in our proof. For notational convenience, let us isolate this property in a separate definition.

\begin{definition}
We will say that a probability vector $\mathbf P=(p_1,p_2,\dots,p_l)$ is \emph{roughly equal} with parameters $\varepsilon>0$, $n\in\na$ and $h>0$ if
$$
\sum_{i\in I}p_i>1-\varepsilon,
$$
where $I=\{i: p_i\in(2^{-n(h+\varepsilon)},2^{-n(h-\varepsilon)})\}$.
\end{definition}
The condition says that majority of the mass of $\mathbf P$ is carried by entries approximately equal to $2^{-nh}$ (hence there is approximately $2^{nh}$ of them). The meaning of ``approximately'' is very rough as the tolerance is up to the (large) multiplicative factor $2^{n\varepsilon}$.

It is a fairly obvious observation that the condition of being roughly equal with some fixed parameters $\varepsilon, n, h$ is open in $\ell^1$.

\begin{definition}
Let $\p=\{P_1,P_2,\dots,P_l\}$ be a finite measurable partition of a standard probability space \xsm. We will say that the partition $\p$ is \emph{$\mu$-roughly equal with parameters $\varepsilon,n,h$} whenever $\mathbf P_{\mu,\p}$ is roughly equal (with the same parameters), 
where $\mathbf P_{\mu,\p}$ denotes the vector of values assigned to $\p$ by the measure $\mu$, i.e.,
$$
\mathbf P_{\mu,\p}=(\mu(P_1),\mu(P_2),\dots,\mu(P_l)).
$$ 
The atoms of $\p$ of measure ranging within $2^{-n(h\pm\varepsilon)}$ will be referred to as \emph{good}.
\end{definition}

We can now rephrase the assertion of the Shannon--McMillan Theorem: for large enough $n$ the partition $\p^{[0,n-1]}$ is $\mu$-roughly equal with parameters $\varepsilon,n,h$. Obviously, by invariance of the measure, the same
holds for any partition of the form $\p^{[a,b-1]}$, where $b-a=n$.

Further, in the same context, if $\mathfrak F$ is some subinvariant sigma-algebra, then the conditional version of the Shannon--McMillan Theorem implies that for a long enough interval $[a,b-1]$ there is a set of measure at least $1-\varepsilon$ of atoms $y$ of $\mathfrak F$ such that the partition $\p^{[a,b-1]}$ is $\mu_y$-roughly equal with parameters $\varepsilon, n, h_\mu(T,\p|\mathfrak F)$, where $\mu_y$  is the disintegration measure of $\mu$ supported by the atom $y$. This conditional version (with a proof, and in the stronger, almost everywhere version) can be found e.g. in Appendix B of \cite{D}.

\medskip
The following fact is completely standard, we leave the proof as an exercise.
\begin{fac}\label{f0}
Let $X'$ be a subset of positive measure of a probability space \xsm. Fix some $\varepsilon$, $h$ 
and $\varepsilon'>\frac{\varepsilon}{\mu(X')}$. Then, for $n$ large enough the following implication
holds: if the partition $\p$ is $\mu$-roughly equal with parameters $\varepsilon, n, h$ then the same partition is 
$\mu_{X'}$-roughly equal with parameters $\varepsilon', n, h$, where $\mu_{X'}$ is the conditional measure of $\mu$ on $X'$. \qed
\end{fac}

\subsection{Conditionally independent partitions}

We assume familiarity of the reader with the notion of Shannon entropy $H(\cdot)$ for probability vectors and $H_\mu(\cdot)$ 
for partitions.
\begin{definition}
Let $\p$ be a finite partition of a probability space \xsm\ and let $\mathfrak F$ be a sub-sigma algebra of $\mathfrak B$.
We will say that $\p$ is \emph{$\delta$-independent} of $\mathfrak F$, (which we will write as $\p\bot^{\!\!\delta}\mathfrak F$), if 
$$
H_\mu(\p|\mathfrak F)>H_\mu(\p)-\delta.
$$
In particular, if $\mathfrak F$ is the algebra generated by a finite partition $\Q$ we will say that $\p$ is \emph{$\delta$-independent} of $\Q$ (and write $\p\bot^{\!\!\delta}\Q$).
\end{definition}
Since we can write the above as $H_\mu(\p\vee\Q)> H_\mu(\p)+H_\mu(\Q)-\delta$, it is immediately seen that the relation $\bot^{\!\!\delta}$ applied to partitions is symmetric. A very important (and easy) fact is that if $\p_1\preccurlyeq\p$ and $\mathfrak F_1\preccurlyeq\mathfrak F$ then $\p\bot^{\!\!\delta}\mathfrak F$ implies $\p_1\bot^{\!\!\delta}\mathfrak F_1$.
Another important fact is that if $(\Q_k)$ is a \sq\ of partitions that generates $\mathfrak F$ and 
$\p\bot^{\!\!\delta}\Q_k$ for every $k$, then $\p\bot^{\!\!\delta}\mathfrak F$ (see \cite{D}, Section 3.1).

\smallskip
What will be very important for us is the interpretation of $\delta$-independence. 
Using disintegration, we can rewrite the definition as
$$
\int H(\mathbf P_{\mu_z,\p})\,d\nu(z) > H(\mathbf P_{\mu,\p})-\delta.
$$
where $\mu_z$ is the disintegration measure of $\mu$ on the atom $z$ of $\mathfrak F$ and $\nu$ is the marginal of $\mu$ 
on $\mathfrak F$. We also have $\mathbf P_{\mu,\p} = \int \mathbf P_{\mu_z,\p}\,d\nu(z)$, hence the above inequality compares a
generalized convex combination of values of the entropy attained on some $l$-dimensional probability vectors 
with the entropy of the corresponding generalized convex combination of these vectors. Because entropy is a continuous 
and strictly convex function on probability vectors of dimension $l$, and the set of all such vectors compact, we can 
deduce that if we fix the partition $\p$ and $\varepsilon>0$ then we can find 
$\delta$ so small that for any sigma-algebra $\mathfrak F$, $\p\bot^{\!\!\delta}\mathfrak F$ implies that 
\begin{equation}\label{b}
\|\mathbf P_{\mu,\p}-\mathbf P_{\mu_z,\p}\|_1<\varepsilon \text{ \ \ for atoms $z$ of $\mathfrak F$ joint measure at least $1-\varepsilon$}
\end{equation}
(here $\|\cdot\|_1$ denotes the norm in $\ell_1$). We refer the reader to Fact 1.1.11 in \cite{D} for a more complete proof. 

Combining this fact with the $\ell^1$ openness of the property of being roughly equal we obtain the following statement:

\begin{fac}\label{f1} Suppose $\p$ is a finite measurable partition of a probability space \xsm, which is $\mu$-roughly equal with 
parameters $\alpha, n, h$. Then for any $\beta>0$ there exists $\gamma>0$ such that whenever a sub-sigma algebra $\mathfrak F$  satisfies $\p\bot^{\!\!\gamma}\mathfrak F$ then $\p$ is $\mu_z$-roughly equal with the same parameters $\alpha, n,h$ and the same 
good atoms, for atoms $z$ of $\mathfrak F$ of joint measure at least $1-\beta$. \qed
\end{fac}


We will need a bit more intricate reversed statement:
\begin{lemma}\label{l1}
Suppose $\p$ is a finite measurable partition of a probability space \xsm, which is $\mu$-roughly equal with parameters $\alpha, n, h$. 
Fix some positive $\beta<\alpha$ and $h'>0$. Then there exist a positive $\gamma<\beta$ and $m$ such that whenever $\Q$ is $\mu$-roughly equal with parameters $\gamma, m, h'$ and $\p\bot^{\!\!\gamma}\Q$ then 
\begin{itemize}
	\item $\p\vee\Q$ is $\mu$-roughly equal with parameters $\alpha+\beta+\gamma, m, h'$, and
	\item for a collection of atoms $A$ of $\p$, of joint measure at least $1-\alpha-\beta$, 
$\Q$ is $\mu_A$-roughly equal with parameters $\beta, m , h'$. 
\end{itemize}
\end{lemma}

\begin{proof}
We choose $\gamma$ according to Fact \ref{f1}, and then, for atoms $B$ of $\Q$ of joint measure at least $1-\beta$, $\p$ is $\mu_B$-roughly equal with parameters $\alpha, n,h$. Moreover, $\p$ has the same good atoms for all such $B$. Let $V'$ denote the 
union of bad atoms $B$ of $\Q$ and atoms $B$ which do not satisfy the above. This set has measure $\mu$ at most $\beta+\gamma$. Every other atom $B$ of $\Q$ splits into a set of measure $\mu_B$ at most $\alpha$ and atoms $A\cap B$ of $\p\vee\Q$ of measures $\mu$ ranging within $2^{-m(h'\pm\gamma)-n(h\pm\alpha)}$. If $m$ is much larger than $n$, this range is contained within $2^{-m(h'\pm2\gamma)}$. The joint measure $\mu$ of these atoms $A\cap B$ is at least $1-\alpha-\beta-\gamma$. This implies (with sacrifice of precision) that $\p\vee\Q$ is $\mu$-roughly equal with the parameters $\alpha+\sqrt{2\beta}+\gamma, m, h'$ (we have artificially enlarged $\beta$ for cosmetic reasons, we will adjust it in a moment).

In what follows, for shorter writing, we replace $\beta+\gamma$ by the larger term $2\beta$. 
It is clear that the collection of atoms $A$ of $\p$ such that $\mu_A(V')>\sqrt{2\beta}$ has jointly the measure $\mu$ smaller than $\sqrt{2\beta}$. Among the remaining atoms $A$ of $\p$ we select only the good ones and call them ``very good''. The joint measure $\mu$ of these atoms $A$ is at least $1-\alpha-\sqrt{2\beta}$. The set $A\setminus V'$ has, for a very good $A$, the measure $\mu_A$ at least $1-\sqrt{2\beta}$ and it consists of good atoms $A\cap B$ (their measures $\mu$ range within $2^{-m(h'\pm\delta)-n(h\pm\alpha)}$). Notice that the measure $\mu_A$ of $A\cap B$ cannot be smaller than its measure $\mu$, while it can be larger by at most the multiplicative factor $2^{n(h+\alpha)}$. So, the measure $\mu_A$ of a good $A\cap B$ ranges between $2^{-m(h'+\gamma)-n(h+\alpha)}$ and $2^{-m(h'-\gamma)+2n\alpha}$, which, for large enough $m$, lies within $2^{-m(h'\pm2\gamma)}$. We obtain (again, with sacrifice of precision) that $\Q$ is $\mu_A$-roughly equal with parameters $\sqrt{2\beta}, m , h'$, and this is true on a subset of good atoms $A$ whose joint measure $\mu$ is at least $1-\alpha-\sqrt{2\beta}$.

Since $\beta$ is arbitrary, we can replace it with $\frac{\beta^2}2$ and we obtain both hypotheses of the lemma.
\end{proof}
\begin{rem}\label{rem}
{\rm Notice in the proof that the intersection of the union of these good atoms of $\p$ on which $\Q$ is roughly equal 
(as specified in the lemma) with the union of good atoms of $\p\vee\Q$ has measure at least $1-\alpha-2\sqrt{2\beta}$.
These atoms alone allow to classify $\p\vee\Q$ as $\mu$-roughly equal with the parameters $\alpha+2\sqrt{2\beta}, m,h'$.
Adjusting $\beta$ and sacrificing some good atoms of $\p\vee\Q$ (treating them as ``not good''), we obtain that the lemma holds with the additional property $V_0\supset V_1\supset V_2$, where $V_0$ is the union of good atoms of $\p$, $V_1$ is the union of 
these atoms of $\p$ on which $\Q$ is roughly equal, and $V_2$ is the union of good atoms of $\p\vee\Q$. We will use this
fact later.}
\end{rem}

\subsection{Pinsker factor}
Let \xsmt\ be some ergodic system with positive entropy $h_\mu(T)$. By $\Pi$ we will denote the Pinsker sigma-algebra of the system. It is obtained as $\bigvee_\p\Pi_\p$, where $\p$ ranges over all finite partitions of $X$ and $\Pi_\p = \bigcap_{n\ge 1}\p^{[n,\infty)}$ ($\Pi_\p$ is sometimes called the \emph{remote future} of the process generated by $\p$). It is known that the sigma-algebra $\Pi$ is invariant (i.e., $T^{-1}(\Pi)=\Pi$) and the factor associated with $\Pi$ has entropy zero (in fact it is the 
largest factor with entropy zero). In particular, for any finite partition $\p$ we have $h_\mu(T,\p|\Pi) = h_\mu(T,\p)$. See 
\cite{D}, sections 3.2 and 4.2 for the proofs and more information.

\begin{lemma} \label{l2}
Let $(X,\mu,T)$ be an ergodic system with positive entropy and let $\p$ and $\Q$ be some finite measurable 
partitions of $X$. Then for every $\varepsilon>0$ there exists $n>1$ such that for every $m>n$ we have 
$$
H_\mu(\p|\Q^{[n,m-1]}\vee\Pi)>H_\mu(\p|\Pi)-\varepsilon.
$$
\end{lemma}

\begin{proof}
The subtlety is that in general it is not true that if $\mathfrak F_n$ is a decreasing \sq\ of sigma-algebras 
and $\mathfrak G$ is another sigma-algebra then $\bigcap_n(\mathfrak G\vee\mathfrak F_n) = \mathfrak G\vee\bigcap_n\mathfrak F_n$. 
However, in our case it will be so; the conditioning sigma-algebras $\Q^{[n,\infty)}\vee\Pi$ decrease to $\Pi$. 
Here is why: the Pinsker factor, as an invertible zero-entropy system always admits a finite bilateral generator $\p_0$ 
(we are using Krieger's Generator Theorem) and then $\Pi = \bigcap_n \p_0^{[n,\infty)}$. Thus 
$$
\Q^{[n,\infty)}\vee\Pi\preccurlyeq\Q^{[n,\infty)}\vee\p_0^{[n,\infty)}=(\Q\vee\p_0)^{[n,\infty)},
$$ 
and
$$
\Pi\preccurlyeq\bigcap_n(\Q^{[n,\infty)}\vee\Pi)\preccurlyeq\bigcap_n(\Q\vee\p_0)^{[n,\infty)}\preccurlyeq\Pi.
$$
Further, since conditional entropy respects decreasing limits of the conditioning sigma-algebras (see \cite{D} 
Fact 1.7.11 (1.7.14)), we have
$$
H_\mu(\p|\Pi) = H_\mu(\p|\bigcap_n(\Q^{[n,\infty)}\vee\Pi))  = \lim_n H_\mu(\p|\Q^{[n,\infty)}\vee\Pi),
$$
which implies that for large enough $n$, $H_\mu(\p|\Q^{[n,\infty)}\vee\Pi)>H_\mu(\p|\Pi)-\varepsilon$.
Obviously, this inequality will be maintained if we replace the sigma-algebra $\Q^{[n,\infty)}$ by the partition $\Q^{[n,m-1]}$.
\end{proof}

\section{The main proof}

\subsection{Measure-theoretic part of the proof}

Consider an arbitrary \tl\ \ds\ \xt\ with positive \tl\ \en. By the variational principle, there exists an ergodic $T$-\inv\ Borel probability measure $\mu$ on $X$ with positive Kolmogorov--Sinai entropy. We fix $\mu$ and we choose a \sq\ of measurable partitions $\p_k$ with the following properties: 
\begin{itemize}
\item the partitions with even indices $2k$ are chosen so that $\p_{2k+2}\succcurlyeq\p_{2k}$ and $\diam(\p_{2k})\to 0$;
\item for all odd indices $2k-1$, $\p_{2k-1}$ equals one fixed finite partition $\p$ which has positive dynamic entropy $h=h_\mu(T,\p)$ and contains one special set $P_0$ of measure smaller than $\delta$ such that if $x,y$ belong to different atoms of $\p$ other than $P_0$ then $\dist(x,y)>t_0$ ($t_0$ and $\delta$ are some positive numbers).
\end{itemize}
Getting a partition $\p$ with the above property (the ``separating atom'' $P_0$) is an easy exercise: One should start with 
an arbitrary partition $\p'$ (having positive dynamical entropy) and then, using regularity of the measure, find inside 
each atom of $\p'$ a compact set of nearly the same measure, and finally define $\p$ as the collection of these compact sets and 
one open set $P_0$ of small measure $\delta$. Now $t_0$ is the minimal distance between the compact members of $\p$. By continuity of the dynamical entropy, it remains positive for $\p$. The above established parameters $t_0$ and $\delta$ will eventually play the role of $t_0$ and the gap between 1 and the lower density in the second requirement in the definition of a scrambled pair. 
Note that the term $\delta$ can be chosen as small as we wish (then $t_0$ will be rather small, too).

We let $h_k$ denote $h_\mu(T,\p_k)$ (of course, $h_k=h$ for odd $k$). Next we will select a strictly increasing \sq\ $S$ of integers starting with $a_1=0$. We will write this sequence as two alternating \sq s, $S=(a_1,b_1,a_2,b_2,a_3,b_3,\dots)$. The \sq\ should grow so fast that the ratios $\frac{b_k}{a_k}$ tend to infinity. Many more conditions on the growth of $S$ will be imposed during the construction. We introduce some more notation. For $k\ge 1$ and $m\ge k$ (including $m=\infty$) we will write 
$$
\R_k = \p_k^{[a_k,b_k-1]},  \text{ \ \ and \ \ }\R_{k,m} = \bigvee_{i=k}^m\R_i 
$$
(for $m=\infty$ the latter is a sigma-algebra rather than a partition). We will also write
$$
\R^{\mathsf{odd}}_{1,2k-1}=\bigvee_{i=1}^k\R_{2i-1}, \ \ \R^{\mathsf{even}}_{2,2k}=\bigvee_{i=1}^k\R_{2i}\text{ \ \ and \ \ }
\mathfrak R = \R^{\mathsf{even}}_{2,\infty}.
$$

Let $(\varepsilon_k)_{k\ge 0}$ be a summable \sq\ of positive numbers with a small sum $\varepsilon_0$. Again, many more conditions
on the decay of this \sq\ will be imposed later. As we have mentioned in the introduction, the conditional Shannon--McMillan Theorem implies that choosing the length $n_k=b_k-a_k$ large enough we can assure that for a family of atoms $y$ of the Pinsker sigma-algebra $\Pi$, of joint measure at least $1-\varepsilon_k$, the partition $\R_k$ is $\mu_y$-roughly equal with parameters $\varepsilon_k, n_k, h_k$. For atoms $y$ from a set $Y$ of measure $1-\varepsilon_0$, this holds for every~$k$.

By our Lemma~\ref{l2}, we can arrange our \sq\ $S$ to grow so fast that for every $k>1$,
\begin{equation}\label{jeden}
H(\R_{1,k-1}|\R_k\vee\Pi)> H(\R_k|\Pi)-\varepsilon^2_k.
\end{equation}
Using disintegration, we can rewrite the above, as
$$
\int H_y(\R_{1,k-1}|\R_k)\,d\nu(y) > \int H_y(\R_{1,k-1})\,d\nu(y)- \varepsilon^2_k,
$$
where $H_y$ is the entropy with respect to the disintegration measure $\mu_y$ of $\mu$ over the atoms $y$ of $\Pi$, and where $\nu$ denotes the marginal of $\mu$ on this sigma-algebra. Because the integrated function on the left is, at every $y$, not larger than that on the right, we conclude that for $y$'s of joint measure at least $1-\varepsilon_k$ we have
\begin{equation}\label{dwa}
H_y(\R_{1,k-1}|\R_k)> H_y(\R_{1,k-1})-\varepsilon_k,
\end{equation}
i.e., $\R_{1,k-1}\bot^{\!\!\varepsilon_k}\R_k$ with respect to the measure $\mu_y$. There exists a set $Y'$ of atoms $y$, of joint
measure at least $1-\varepsilon_0$, where the above holds for all $k$. Of course, $Y\cap Y'$ has measure at least $1-2\varepsilon_0$ which is positive. We fix an atom $y\in Y\cap Y'$ (the fact the $y$ belongs to $Y$ will be used later). 

By elementary identities concerning conditional entropy, we have, for $m> k$,
\begin{multline*}
H_y(\R_{1,m}) = \\
H_y(\R_{1,m-1}|\R_m)+H_y(\R_m)> H_y(\R_{1,m-1})+H_y(\R_m) -\varepsilon_m =\\
H_y(\R_{1,m-2}|\R_{m-1})+H_y(\R_{m-1})+ H_y(\R_m) -\varepsilon_m > \\
H_y(\R_{1,m-2})+ H_y(\R_{m-1})+ H_y(\R_m)-\varepsilon_{m-1}-\varepsilon_m = \\
\vdots \\ 
H_y(\R_{1,k-1})+H_y(\R_k)+H_y(\R_{k+1})+ \dots + H_y(\R_m)-\varepsilon_k-\varepsilon_{k+1}-\dots-\varepsilon_m.
\end{multline*}
Denoting by $\xi^2_k$ the tail $\sum_{i\ge k}\varepsilon_k$, we draw two conclusions:
$$
H_y(\R_{1,m}) = H_y(\R_{1,k-1}\vee\R_{k,m}) \ge H_y(\R_{1,k-1})+H(\R_{k,m})-\xi_k^2.
$$
i.e., $\R_{1,k-1}\bot^{\!\!\xi^2_k}\R_{k,m}$ with respect to $\mu_y$,
and
$$
H_y(\R_{1,m}) = H_y(\R_k\vee(\R_{1,k-1}\vee\R_{k+1,m})) \ge H_y(\R_k)+H_y(\R_{1,k-1}\vee\R_{k+1,m})-\xi_k^2,
$$
i.e., $\R_k\bot^{\!\!\xi^2_k}(\R_{1,k-1}\vee\R_{k+1,m})$ with respect to $\mu_y$.
Because this holds for all $m> k$, we get
\begin{equation}\label{a}
\R_{1,k-1}\bot^{\!\!\xi^2_k}\R_{k,\infty} \text{ \ \ with respect to $\mu_y$}
\end{equation}
(which we will use a bit later) and $\R_k\bot^{\!\!\xi^2_k}(\R_{1,k-1}\vee\R_{k+1,\infty})$. 
Applying the latter to $2k+1$ rather than $k$ and using the fact that the independence passes to sub-sigma algebras, we 
obtain that $\R_{2k+1}\bot^{\!\!\xi^2_{2k+1}}(\R^{\mathsf{odd}}_{1,2k-1}\vee\mathfrak R)$ with respect to $\mu_y$,  
i.e., 
$$
H_y(\R_{2k+1}|\R^{\mathsf{odd}}_{1,2k-1}\vee\mathfrak R)>H_y(\R_{2k+1})-\xi^2_{2k+1}, 
$$
which obviously implies
\begin{equation}\label{trzy}
H_y(\R_{2k+1}|\R^{\mathsf{odd}}_{1,2k-1}\vee\mathfrak R)>H_y(\R_{2k+1}|\mathfrak R)-\xi^2_{2k+1}.
\end{equation}
Exactly the same calculation as we used to pass from \eqref{jeden} to \eqref{dwa} allows us to deduce from \eqref{trzy} that
there exists a set of atoms $z$ of the sigma-algebra $\mathfrak R$, of joint measure $\mu_y$ at least $1-\xi_{2k+1}$, such that
\begin{equation}\label{cztery}
H_{yz}(\R_{2k+1}|\R^{\mathsf{odd}}_{1,2k-1})>H_{yz}(\R_{2k+1})-\xi_{2k+1},
\end{equation}
i.e., $\R^{\mathsf{odd}}_{1,2k-1}\bot^{\!\!\xi_{2k+1}}\R_{2k+1}$ with regard to the measure $\mu_{yz}$ (the disintegration measure of $\mu_y$ with respect to $\mathfrak R$ on the atom $z$, the same as the disintegration of $\mu$ with respect to $\Pi\vee\mathfrak R$ at $y\cap z$). We can require that the \sq\ $\xi_k$ is summable to a number smaller than $\frac12$ (this holds for sufficiently fast decreasing parameters $\varepsilon_k$) and then there exists a set $Z$ of measure $\nu_y$ larger than $\frac12$ of atoms $z$ of $\mathfrak R$ satisfying for all $k$, 
\begin{enumerate}[(a)]
	\item   
	$\R^{\mathsf{odd}}_{1,2k-1}\bot^{\!\!\xi_{2k+1}}\R_{2k+1}$ with regard to the measure $\mu_{yz}$. 
\end{enumerate}

Notice the obvious fact, that (a) holds if we replace the \sq\ of partitions $(\R_k)$ by any subsequence
$(\R_{j_k})$ such that $k\to j_k$ preserves parity. By passing to a subsequence we can control the speed of growth of 
the parameters $n_k$ and of decay of the parameters $\varepsilon_k$. 
\medskip

Because $y\in Y$, each partition $\R_{2k}$ is $\mu_y$-roughly equal with the parameters $\varepsilon_{2k}, n_{2k}, h_{2k}$, and it follows from \eqref{dwa} that with respect to $\mu_y$ it is $\varepsilon_{2k}$-independent of the preceding even-indexed partitions. 

We are now going to apply the first part of Lemma \ref{l1} inductively (for the first time; we will do it again, later). For $k=1$
we apply it with $\alpha=\varepsilon_2, \beta=\varepsilon_3, \gamma=\varepsilon_4, n=n_2$ and $m=n_4$ and obtain that if $\varepsilon_4$ is small and $n_4$ large enough (this can be achieved by passing to a subsequence of $(\R_k)$) then $\R_2\vee\R_4$ is $\mu_y$-roughly equal with the parameters $\varepsilon_2+\varepsilon_3+\varepsilon_4, n_4, h_4$. Repeating this argument we 
easily deduce that, by passing to a sub\sq, we can arrange $\R^{\mathsf{even}}_{2,2k}$ to be $\mu_y$-roughly equal with parameters $\varepsilon_2+\varepsilon_3+\dots+\varepsilon_{2k}$, $n_{2k}$, $h_{2k}$, for every $k$. We denote the first parameter (i.e., $\sum_{i=2}^{2k}\varepsilon_i$) by $s_k$. The following three properties (which we have obtained): 
\begin{enumerate}
	\item $\R^{\mathsf{even}}_{2,2k}$ is $\mu_y$-roughly equal with parameters $s_k$, $n_{2k}$, $h_{2k}$,
	\item $\R_{2k+1}$ is $\mu_y$-roughly equal with parameters $\varepsilon_{2k+1}$, $n_{2k+1}$, $h$,
	\item $\R^{\mathsf{even}}_{2,2k}\bot^{\!\!\varepsilon_{2k+1}}\R_{2k+1}$ with respect to $\mu_y$ (a consequence of \eqref{dwa}),
\end{enumerate}
are open properties of the probability vector assigned by $\mu_y$ to the finite partition $\R_{1,2k+1}$.
It follows from \eqref{a} that $\R_{1,2k+1}\bot^{\!\!\xi^2_{2k+2}}(\R_{2k+2}\vee\R_{2k+4}\vee\dots)$ with respect to $\mu_y$, so, by \eqref{b}, if we arrange $\xi^2_{2k+2}$ small enough, we will have these properties on a large (say, larger than $1-\varepsilon_{2k}$) 
set of atoms $v$ of $(\R_{2k+2}\vee\R_{2k+4}\vee\dots)$:
\begin{enumerate}
	\item $\R^{\mathsf{even}}_{2,2k}$ is $\mu_{yv}$-roughly equal with parameters $s_k$, $n_{2k}$, $h_{2k}$,
	\item $\R_{2k+1}$ is $\mu_{yv}$-roughly equal with parameters $\varepsilon_{2k+1}$, $n_{2k+1}$, $h$,
	\item $\R^{\mathsf{even}}_{2,2k}\bot^{\!\!\varepsilon_{2k+1}}\R_{2k+1}$ with respect to $\mu_{yv}$.
\end{enumerate}
We can now apply the second part of Lemma \ref{l1} (with $\alpha=s_k,\beta=\varepsilon_{2k}, \gamma=\varepsilon_{2k+1}$) 
to deduce that (with a good choice of $\varepsilon_{2k+1}$), $\R_{2k+1}$ is $\mu_{yvu}$-roughly equal with the parameters $\varepsilon_{2k}$, $n_{2k+1}, h$ on a set of atoms $u$ of $\R^{\mathsf{even}}_{2,2k}$ of joint measure $\mu_{yv}$ at least  
$1-s_k-\varepsilon_{2k}$ (which is larger than $1-\varepsilon_0$). 
Notice that the intersection $v\cap u$ is an atom $z$ of $\mathfrak R$. In this manner we have derived:
\begin{enumerate}[(b)]
\item $\R_{2k+1}$ is $\mu_{yz}$-roughly equal with parameters $\varepsilon_{2k}$, $n_{2k+1}, h$
on a set of atoms $z$ of $\mathfrak R$ of joint measure $\mu_y$ larger than $1-\varepsilon_0$.
\end{enumerate}
It is now obvious that there exists a set $Z'$ of measure $\mu_y$ larger than $1-\varepsilon_0$
consisting of atoms $z$ of $\mathfrak R$ satisfying (b) for infinitely many indices $k$. 
We choose an atom $z$ from the nonempty intersection $Z\cap Z'$.

\medskip
At this point we intrude our proof by an ingredient needed for the topological requirements. We address the ``separating set'' $P_0$. By the Ergodic Theorem, we can arrange the lengths $n_k$ so large that for a set of points $x\in X$ of measure $1-\varepsilon_k$, $T^nx$ visits $P_0$ for no more than $n_k\delta$ indices $n\in[a_k,b_k-1]$. In a set $X'$ of measure at least $1-\varepsilon_0$, this happens for every $k$. It is clear that the collection of atoms $y\cap z$ of $\Pi\vee\mathfrak R$ for which $\mu_{yz}(X')<1-\sqrt{\varepsilon_0}$ has joint measure at most $\sqrt{\varepsilon_0}$. Thus, for $\varepsilon_0$ small enough, we can adjust the choices of $y\in Y\cap Y'$ 
and $z\in Z\cap Z'$ so that $\mu_{yz}(X')\ge1-\sqrt{\varepsilon_0}$. Now we can use our Fact \ref{f0}. If we replace $\mu_{yz}$
by its conditional measure on $X'$ then for every $k$, $\R_{2k+1}$ remains roughly equal with the parameters $\varepsilon'_{2k+1}$, $n_{2k+1}, h$ (where $\varepsilon'_{2k+1}$ is slightly larger than $\frac{\varepsilon_{2k+1}}{1-\sqrt{\varepsilon_0}}$). 
By a an adjustment of the parameters we will skip the ``prime'' and still write $\varepsilon_{2k+1}$. We will also write $\mu_{yz}$ (rather than $\mu_{yzX'})$ remembering that we are dealing with the restriction to $X'$. 
It is also easy to see that this restriction does not affect the independences arranged in (a) (again, a small adjustment of $\varepsilon_{2k+1}$ does the job). From now on we will think that all our points belong to the atom $y\cap z$ (which supports of the conditional measure $\mu_{yz}$) intersected with $X'$, in particular, all our partitions are treated as partitions of $y\cap z\cap X'$. 

\medskip
We would like to pass to the subsequence of $S$ along which (b) is fulfilled for $z$. This time, however, we cannot affect the choice of $z$ (and thus neither the construction of $\mathfrak R$). So, we keep all even-indexed partitions $\R_{2k}$ and we select the sub\sq\ only from the odd-indexed partitions. In this manner we can assume that the atom $z$ satisfies both (a) and (b) for all $k$. 

We shall use inductively the first part of Lemma 1 for the second time, now with respect to the measure 
$\mu_{yz}$. In order to have the same parameter in (a) and (b), we set $\delta_k = \max\{\xi_k,\varepsilon_{k-1}\}$
and we require that these numbers form a summable \sq\ with a small sum $\delta_0$. Now we have an almost identical situation 
as in the first inductive application, except that we look at the odd-indexed (rather than even-indexed) partitions $\R_{2k+1}$ and 
we have the parameters $\delta_{2k+1}$ in place of the former $\varepsilon_{2k}$. As a result we get that, by passing to 
a subsequence, we can arrange each of the partitions $\R_{1,2k+1}^{\mathsf{odd}}$ to be $\mu_{yz}$-roughly equal 
with the parameters $r_k, n_{2k+1},h$, where $r_k = \sum_{i=1}^{2k+1}\delta_k$.

The second part of Lemma 1 (with $\alpha=r_{k-1},\beta=\delta_{2k}$ and $\gamma = \delta_{2k+1}$) now gives: 
on atoms $B$ of the partition $\R_{1,2k-1}^{\mathsf{odd}}$, of joint measure $\mu_{yz}$ at least $1-r_k-\delta_{2k}$ (which is 
larger than $1-2\delta_0$), the partition $\R_{2k+1}$ is $\mu_{yzB}$-roughly equal with the parameters $\delta_{2k}$, $n_{2k+1}, h$.
Let $V_k$ denote the union of the good atoms $B$ for which this holds. Using Remark \ref{rem}, we can arrange that the sets $V_k$ decrease.

We can now draw the following conclusion ($h'$ denotes a constant is slightly smaller than $h$): 
\begin{itemize}
	\item If $B$ is an atom of $\R_{1,2k-1}^{\mathsf{odd}}$, contained in $V_k$, then it contains at least $2^{n_{2k+1}h'}$ (nonempty) atoms of $\R_{1,2k+1}^{\mathsf{odd}}$ contained in $V_{k+1}$.
\end{itemize}

\subsection{Counting Hamming balls}\label{sect}

Let $l$ denote the cardinality of $\p$. In the symbolic representation of the process $(X,\p,\mu,T)$, the partition 
$\R_{2k+1}$ depends on the interval of coordinates $[a_{2k+1},b_{2k+1}-1]$ of length $n_{2k+1}$ and its atoms 
can be identified with blocks over the alphabet $\p$. The \emph{Hamming distance} between two such 
blocks is the number of positions where the blocks disagree, divided by $n_{2k+1}$. 

It is a standard fact that any ball of radius $\delta$ in the Hamming distance around a block 
of length $n$ has at most $2^{n(H(\delta,1-\delta)+\delta\log l)}$ elements (other blocks). 
For a small enough $\delta$ the term $\beta(\delta)=H(3\delta,1-3\delta)+3\delta\log l$ is smaller than $h'$
(the constants $h'$ and $\delta$ depend only on $h$ and they and can be established at the start of the proof). 

\medskip
We will now indicate a family $B_\kappa$, where $\kappa$ ranges over all finite binary blocks, such that $B_\kappa$
is a (nonempty) atom of $\R^{\mathsf{odd}}_{1,2k+1}$ contained in $V_{k+1}$, where $k$ is the length of $\kappa$ (for $k=0$, $\kappa$ 
is the empty word). We will assure that if $\iota$ extends $\kappa$ to the right then $B_{\iota}\subset B_\kappa$.
Moreover, we will arrange that if $\kappa\neq\kappa'$, both have the same length $k$, and differ at a position $i_0$, 
then the Hamming distance between $A_i$ and $A_i'$ is at least $3\delta$ for all $i=i_0,i_0+1,\dots,k$, where $A_i$ 
and $A_i'$ are the blocks corresponding to the coordinates $[a_{2i+1},b_{2i+1-1}]$ appearing in the 
representation of the atoms $B_\kappa$ and $B_{\kappa'}$, respectively. We will do it by induction on $k$, in each step 
we choose two ``children'' of every so far constructed atom of $\R^{\mathsf{odd}}_{1,2k-1}$. 

In step $k=0$ we assign $B_\emptyset$ to be an arbitrarily selected atom of $\R_1$ contained in $V_1$. 
Suppose the task has been completed for some $k-1$, i.e., that we have selected $2^{k-1}$ atoms $B_\kappa$ 
of $\R^{\mathsf{odd}}_{1,2k-1}$, contained in $V_{k-1}$, and pairwise $3\delta$-separated as required. We know that 
every such $B_\kappa$ contains at least $2^{n_{2k+1}h'}$ atoms of $\R_{1,2k+1}^{\mathsf{odd}}$ contained in $V_{k+1}$. 
Every such atom has the form $B_\kappa\cap A$ where $A$ is an atom of $\R_{2k+1}$. We will call the atoms $A$ such 
that $B_\kappa \cap A$ is contained in $V$ \emph{good continuations} of $B_\kappa$. 

Let $\mathcal A(B_\kappa)$ denote the family of all good continuations of $B_\kappa$. Take the first selected atom
$B_{\kappa_1}$ (for $\kappa_1=000\dots0$). Choose one good continuation $A_0$ of $B_{\kappa_1}$. From every family
$\mathcal A(B_\kappa)$ (including $\kappa=\kappa_1$) we eliminate (for future choices) all the atoms $A$ that belong 
to the Hamming ball of radius $3\delta$ around $A_0$. Every family $\mathcal A(B_\kappa)$ has ``lost'' at most $2^{n_{2k+1}\beta(\delta)}$ elements (out of $2^{n_{2k+1}h'}$), so vast majority remains. Next we choose $A_1$ from the remaining good continuations of $B_{\kappa_1}$ and again, from each of the families $\mathcal A(B_\kappa)$ we eliminate all the atoms $A$ that belong 
to the Hamming ball of radius $\delta$ around $A_1$. Again, the losses are negligibly small. We assign 
$B_{\kappa_10}=B_{\kappa_1}\cap A_0$ and $B_{\kappa_11}=B_{\kappa_1}\cap A_1$ (here $\kappa_10$ and $\kappa_11$ denote the
two continuations of $\kappa_1$).

Next we abandon $B_{\kappa_1}$ and pass to $B_{\kappa_2}$ (earlier we need to order the $\kappa$'s somehow) and we repeat the procedure choosing two of its good continuations, say $A'_0,A'_1$, not eliminated in the preceding steps, each time eliminating for future 
choices all members of the respective Hamming balls. We proceed until we choose two good continuations for every $\kappa$ of length $k$. Note that near the end of this procedure we will have eliminated less than $2^{k+n_{2k+1}\beta(\delta)}$ blocks from each family $\mathcal A(B_\kappa)$, which is still a negligible fraction of the cardinality of $\mathcal A(B_\kappa)$, hence the procedure will 
be possible till the end.

\subsection{Verification of DC2} 

From now on we use the letter $\kappa$ to denote infinite binary \sq s. 
For every $\kappa$ we define a set $C_\kappa$ as the decreasing intersection 
$\bigcap_k\overline{B_{\kappa_k}}$, where $\kappa_k$ is the initial string of length $k$ of $\kappa$, and the closure is taken in $X$ (the sets $B_{\kappa_k}$ are understood as subsets of $y\cap z\cap X'$, still, we close them in the entire compact space $X$). In this manner we have defined a family of nonempty closed sets indexed by the binary \sq s. The scrambled set $E$ is obtained by selecting one point from every set $C_\kappa$ (a priori we have no guarantee that the sets $C_\kappa$ are pairwise disjoint, but this will become
obvious if we show that every pair of points in $E$ is scrambled). 

Take two points, $x_1,x_2$ from $E$. They belong to the closure of the same atom $z$ of $\mathfrak R$ (also to the closure of the 
same atom $y$ of $\Pi$, but this fact has no importance). Any pair of points $x,x'$ belonging to $z$ (without the closure) satisfy $\dist(T^nx,T^nx')\le\varepsilon_k$ for all $k$ and $n\in [a_{2k},b_{2k}-1]$. This property obviously passes to all pairs in the closure of the atom $z$, in particular to $x_1,x_2$. Since $\frac{b_{2k}}{a_{2k}}$ tends to infinity, it is very easy to see that for every $t>0$ the orbits of our two points are at least $t$ apart along a set of times of upper density~1. We have shown that $E$ satisfies the first requirement of being DC2-scrambled.

\smallskip
On the other hand, our two points belong to the sets $C_\kappa$ for two $\kappa$'s that differ at some position $i_0$.
This implies that for each $k\ge i_0$ they belong to the closures of two atoms $B_{\kappa_k}$, $B_{\kappa'_k}$  where
$\kappa_k$ and $\kappa_k'$ are binary strings of length $k$ differing at $i_0$. We will first analyze the behavior of a pair 
$x,x'$ belonging to the atoms $B_{\kappa_k}$, $B_{\kappa'_k}$, respectively (without the closures). By the construction, we 
know that the Hamming distance between the blocks $A_{\kappa_k}$ and $A_{{\kappa'}_{\!k}}$, representing the atoms of $\R_{2k+1}$
to which belong $x$ and $x'$, is at least $3\delta$. This means that $T^nx,T^nx'$ belong to different atoms of $\p$ for
at least $n_{2k+1}3\delta$ times $n\in [a_{2k+1},b_{2k+1}-1]$. Because we treat these atoms as subsets of $X'$, the situation, 
where at least one of the aforementioned atoms of $\p$ is $P_0$ may happen at most $n_{2k+1}2\delta$ times, leaving at least 
$n_{2k+1}\delta$ times $n$ when the two points fall in two different atoms of $\p$, different from $P_0$. By the property of $P_0$, 
then the distance between these points is at least $t_0$. Now observe that the property 
\begin{itemize}
\item $\dist(T^nx,T^nx')\ge t_0$ for at least $n_{2k+1}\delta$ times $n\in[a_{2k+1},b_{2k+1}-1]$
\end{itemize}
is closed, hence it passes to the pairs of points selected from the closures of the respective atoms. 
We deduce that our points $x_1$, $x_2$ have this property for every $k\ge i_0$. Since $\frac{b_{2k+1}}{a_{2k+1}}$ tends to infinity, this easily proves that 
$\dist(T^nx,T^nx')\ge t_0$ with upper density at least $\delta$. In particular, the points $x_1,x_2$ satisfy the second requirement in the definition of a DC2-scrambled pair, end of the proof.
\qed 

\section{Final remarks}

\begin{rem}{\rm Notice that our scrambled set has additional properties: the parameters of scrambling (i.e., the number $t_0$
and the lower density $\delta$) are common for all pairs (which is good, in the sense that the chaos is strong). Such situation
is called \emph{uniform DC2} (see \cite{BS}, see also \emph{uniform DC1} in \cite{O2}).

On the other hand, the scrambling is ``synchronic'', i.e., all points in the scrambled set get together at the same times (at least
we do not control how often they get together pairwise at other times). Such behavior is somewhat too regular to satisfy what one intuitively expects from ``chaos'' (one expects some kind of independence in the behavior of points). But... this is how chaos is defined -- the definition admits ``synchronic chaos''. It might be interesting to invent and study a stronger notion of chaos which imposes (with positive upper density) some pairs to be together at the same times 
as other pairs are apart. The author is aware that such definitions are around the corner. In particular, one might ask whether 
such a strengthened chaos follows from positive entropy.}
\end{rem}

\begin{rem}{\it This remark is due to Piotr Oprocha. The author thanks him for the contribution. 
\rm Sm\'ital also conjectured that in the positive entropy case the DC2-scrambled set can be obtained perfect (see e.g. \cite{S2}). In fact, we can prove this conjecture as well. 
The sets $C_\kappa$ constructed in Section \ref{sect} are disjoint (this follows from the separation by the distance $t_0$ proved later), compact and have a compact union. It is also seen from the construction, that ``belonging in the same $C_\kappa$'' is a Borel equivalence relation on the union. A special case of Silver's Theorem (see \cite{Sr} Theorem 5.13.9) implies that if the classes of a Borel equivalence relation are of type $F_\sigma$ and there are uncountably many of them then there exists a perfect set having at most one element in every class. It is easy to see (by the rules of separation) that this perfect set is a Cantor set.}
\end{rem}

\bigskip\noindent
Institute of Mathematics and Computer Science, Wroclaw University of Technology, Wybrze\.ze Wyspia\'nskiego 27, 50-370 Wroc\l aw, Poland

\noindent
downar@pwr.wroc.pl


\begin{thebibliography}{99999999}

\bibitem[BS]{BS}
  M. Babilonov\'a- \v Stef\'ankov\'a, 
  \emph{Extreme chaos and transitivity}, Internat. J. Bifur. Chaos Appl. Sci. Engrg. {\bf 13}~(2003) 1695--1700.

\bibitem[BHS]{BHS}
   F. Blanchard, W. Huang, and L. Snoha,
  \emph{Topological size of scrambled sets}, Colloq. Math. {\bf 110}~(2008), 293--361

\bibitem[BGKM]{BGKM}
   F. Blanchard, E. Glasner, S. Kolyada, and A. Maass,
  \emph{On Li-Yorke pairs}, J. Reine Angew. Math. {\bf 547}~(2002), 51--68

\bibitem[D]{D}
  T. Downarowicz,
  {\bf Entropy in dynamical systems}, 
  Cambridge University Press, {\it New Mathematical Monographs 18}, Cambridge 2011

\bibitem[LY]{LY}
  T. Y. Li and J. A. Yorke,
  \emph{Period three implies chaos}, Amer. Math. Monthly {\bf 82}~(1975), 985--992

\bibitem[O1]{O1}
  P. Oprocha,
  \emph{Minimal systems and distributionally scrambled sets}, Bull. S.M.F., to appear

\bibitem[O2]{O2}
  P. Oprocha,
  \emph{Distributional chaos revisited}, Trans. Amer. Math. Soc. {\bf 361}~(2009), 4901--4925.

\bibitem[P]{P}
  R. Pikula,
  \emph{On some notions of chaos in dimension zero}, 
  Colloq. Math. {\bf 107}~(2007), 167-177

\bibitem[SS]{SS}
  B. Schweizer and J. Sm\'ital
  \emph{Measures of chaos and a spectral decomposition of dynamical systems on the interval}, 
  Trans. Amer. Math. Soc. {\bf 344}~(1994), 737--754

\bibitem[S1]{S1}
  J. Sm\'ital
  \emph{Chaotic functions with zero topological entropy}, 
  Trans. Amer. Math. Soc. {\bf 297}~(1986), 269--282
  
\bibitem[S2]{S2}
  J. Sm\'ital 
  \emph{Distributional chaos and topological entropy}, Real Analysis Exchange, Summer Symposium 2006, 61--66

\bibitem[Sr]{Sr}
  S.M. Srivastava,
  {\bf A Course on Borel Sets}, Springer-Verlag, New York Berlin Heidelberg, 1998
\end{thebibliography}
\end{document}